\documentclass[10pt]{article}
\usepackage[english]{babel}										% Idioma a ser usado
\usepackage[utf8]{inputenc}										% Escrita de caracteres acentuados e cedilhas - 1
\usepackage[T1]{fontenc}										% Escrita de caracteres acentuados + outros detalhes técnicos fundamentais
\usepackage{comment}
\usepackage{amsmath,amsfonts,amssymb,amsthm,latexsym}

\usepackage{authblk}

\newtheorem{theorem}{Theorem}[]
\newtheorem{proposition}[theorem]{Proposition}

\theoremstyle{definition}

\newtheorem{example}[theorem]{Example}

\newtheorem{remark}[theorem]{Remark}

\newcommand{\Hol}{\mathrm{Hol}}

\newcommand{\Ker}{\operatorname{Ker}}

\newcommand{\GL}{\mathrm{GL}}

\newcommand{\Cent}{\operatorname{Cent}}
\newcommand{\End}{\operatorname{End}}

\newcommand{\Aut}{\operatorname{Aut}}

\newcommand{\Id}{\operatorname{Id}}

\newcommand{\Z}{\mathbf Z}

\newcommand{\F}{\mathbf F}

\newcommand\blfootnote[1]{%
  \begingroup
  \renewcommand\thefootnote{}\footnote{#1}%
  \addtocounter{footnote}{-1}%
  \endgroup
}

\usepackage{multicol,multirow,tabularx,array}          % Packages para tabela
\title{Left braces of size $8p$}

\begin{document}
\author[1]{Teresa Crespo}
\author[2]{Daniel Gil-Mu\~noz}
\author[3]{Anna Rio}
\author[3]{Montserrat Vela}

\affil[1]{\footnotesize Departament de Matemàtiques i Informàtica, Universitat de Barcelona, Gran Via de les Corts Catalanes 585, 08007, Barcelona (Spain)}
\affil[2]{\footnotesize Charles University, Faculty of Mathematics and Physics, Department of Algebra, Sokolovska 83, 18600 Praha 8, Czech Republic}
\affil[3]{\footnotesize Departament de Matem\`atiques, Universitat Polit\`ecnica de Catalunya, Edifici Omega, Jordi Girona, 1-3, 08034, Barcelona (Spain)}

\begin{comment}

\author{%
Anna Rio, %$^{*}$ %\ and \ Second-Author-Name$^{2}$
Teresa Crespo, Daniel Gil-Muñoz and  Montserrat Vela}
% -------------------------------------------------------------------
\begin{center}
{\footnotesize
E-mails:ana.rio@upc.edu, teresa.crespo@ub.edu, daniel.gil-munoz@mff.cuni.cz, montse.vela@upc.edu
}
\end{center}
%Início do documento

\date{\today}

\vspace{-0.5cm}

\begin{center}
{\footnotesize
%*Corresponding author\\
%$^1$
Dep. Matemàtiques. Universitat Politècnica de Catalunya \\
$^2$Address (second author) \\ Departament de Matemàtiques i Informàtica, Universitat de Barcelona, Gran Via de les Corts Catalanes 585, 08007, Barcelona (Spain)\\
E-mails:ana.rio@upc.edu, teresa.crespo@ub.edu, montse.vela@upc.edu %/ email@2
}
\end{center}
\end{comment}

\maketitle

% -------------------------------------------------------------------
% Abstract
\bigskip
\noindent
{\small{\bf ABSTRACT.}
We describe all left braces of size $8p$ for an odd prime $p\ne 3,7$ and validate the number given by Bardakov, Neschadim and Yadav in \cite{BNY}. We give a characterization for isomorphism classes of a semidirect product of left braces and then the description is done by first describing left braces of size $8$, as conjugacy classes of regular subgroups of the corresponding holomorph, and then checking how many non isomorphic left braces of size $8p$ are obtained from each one of them.
}

\medskip
\noindent
{\small{\bf 2020 MSC:} 16T05; 16T25; 20B35; 81R50.}

{\small{\bf Keywords}{:}
Braces, Hopf-Galois extensions, Holomorphs.
}

\baselineskip=\normalbaselineskip
% -------------------------------------------------------------------

\blfootnote{The first author was supported
by grant PID2019-107297GB-I00 (Ministerio de Ciencia, Innovación y Universidades).
The second author was supported by Czech Science Foundation, grant 21-00420M, and by Charles University Research Centre program UNCE/SCI/022.

Email addresses: teresa.crespo@ub.edu, daniel.gil-munoz@mff.cuni.cz, ana.rio@upc.edu, montse.vela@upc.edu }

\section{Introduction}\label{sec:1}

In \cite{Rump} Rump introduced braces to study
set-theoretic solutions of the Yang-Baxter equation. A left brace is a set $B$ with two operations $+$ and $\cdot$ such
that $(B, +)$ is an abelian group, $(B, \cdot)$ is a group and
$$a(b+c)+a = ab+ac,$$
for all $a, b, c \in B$. We call $N=(B, +$) the additive group and $G=(B, \cdot)$ the multiplicative group of the left brace.

Let $B_1$ and $B_2$ be left braces. A map $f : B_1 \to B_2$ is said to be a brace homomorphism if $f(b+b') = f(b)+f(b')$ and $f(bb') = f(b)f(b')$ for all $b,b' \in B_1$. If $f$ is bijective, we say that $f$ is an isomorphism. In that case we say that the braces $B_1$ and $B_2$ are isomorphic.

This gives the notions of brace isomorphism and isomorphic left braces.

In \cite{Bachiller} %prop2.3
Bachiller proved that given an abelian group $N$, there is
a bijective correspondence between left braces with additive group $N$, and regular subgroups of $\Hol(N)$ such that isomorphic left braces correspond to conjugate subgroups of $\Hol(N)$ by elements of $\Aut(N)$. In this way he established the connection between braces and Hopf-Galois separable extensions.

In \cite{BNY}, Lemma 2.1,  it is proved that $\Aut(N)$, as a subgroup of $\Hol(N)$, is action-closed with respect to the conjugation action of $\Hol(N)$ on
the set of regular subgroups of $\Hol(N)$. Therefore, given an abelian group
$N$, the non-isomorphic left braces with additive group $N$ are in bijective correspondence with conjugacy classes of regular subgroups in $\Hol(N)$.
In \cite[Conjecture 4.2]{BNY}, %that paper, the authors
Bardakov, Neschadim and Yadav conjectured the number $b(8p)$ of left braces of size $8p$ for $p\ge 11$ a prime number:
$$b(8p)= \left\{ \begin{array}{ll} 90 & \text{\ if } p\equiv 3, 7 \pmod{8}, \\
	106 & \text{\ if } p\equiv 5 \pmod{8}, \\
	108 & \text{\ if } p\equiv 1 \pmod{8}.
	\end{array} \right.
	$$
Our aim is to describe all the isomorphism classes of braces of size $8p$ in order to check the validity of this conjecture.

\section{Braces of size $8p$}

The theory of braces mimics many of the constructions and definitions of group theory (see \cite{Cedo}). If $p=5$ or $p\ge 11$, the Sylow $p$-subgroup of a group of order $8p$ is a normal subgroup and therefore the group is a direct or semidirect product of the (unique) group of order $p$ and a group of order $8$.
Our aim is to prove that we have the same situation for braces. In order to do that, let us define direct and semidirect product of braces as in \cite{Cedo} or \cite{SV}.

Let $B_1$ and $B_2$ be left braces. Then $B_1 \times B_2$ together with
$$
(a,b)+(a',b')=(a+a',b+b')\quad (a,b)\cdot(a',b')=(aa',bb')
$$
is a left brace called the direct product of braces $B_1$ and $B_2$.

%Let $B_1$ and $B_2$ be left braces.
Now, let $\tau:(B_2,\cdot)\to\Aut(B_1,+,\cdot)$
be a homomorphism of groups.
Consider in $B_1\times B_2$
the additive structure of the direct product $(B_1,+)\times (B_2,+)$
$$
(a,b)+(a',b')=(a+a',b+b')
$$
and the multiplicative structure of the semidirect product
$(B_1,\cdot)\rtimes_{\tau} (B_2,\cdot)$
$$
(a,b)\cdot(a',b')=(a\tau_b(a'),b b')
$$
Then, we get a left brace,  which is called the semidirect product of the left braces $B_1$ and $B_2$ via $\tau$.

From \cite{SV} we know that if $N$ is the additive group of a brace and $N=N_1\times \dots \times N_k$ is its Sylow decomposition, then every $N_i$ is also the additive group of a brace.

If $p$ is an odd prime and $N$ is an abelian group of size $8p$, then $N$ has Sylow %subgroup
decomposition $N=\Z_p\times E$, where $E$ is an abelian group of order $8$. For the simple group $\Z_p$ we have just the trivial brace, namely the multiplicative group is also $\Z_p$ (we can use also the notation $C_p$). For the abelian group of order $8$ we can have several multiplicative groups giving a left brace structure.

\begin{proposition}\label{braceprod} Let $p=5$ or $p\ge 11$ be a prime.
Every left brace of size $8p$ is a direct or semidirect product of the trivial brace of size $p$ and a left brace  of size $8$.
\end{proposition}

\begin{proof}
Let $B$ be a left brace of size $8p$ with additive group $N$ and multiplicative group $G$. Then, $N=\Z_p\times E$ with $E$ abelian of order $8$ and $G=\Z_p\rtimes_{\tau} F$ with $F$ a group of order $8$ and $\tau:F\to\Aut(\Z_p)$ a group homomorphism (the trivial one giving the direct product). Let us observe that, since we are working with the trivial brace, the group of brace automorphisms is the classical group $\Aut(\Z_p)\simeq Z_p^*$. %image of generator

Then,
$$\begin{matrix}
(a_1,a_2)((b_1,b_2)+(c_1,c_2))+(a_1,a_2)=
(a_1,a_2)(b_1+c_1,b_2+c_2)+(a_1,a_2)=\\
=(a_1+\tau_{a_2}(b_1+c_1)+a_1, a_2(b_2+c_2)+a_2).
\end{matrix}
$$
On the other hand,
$$\begin{matrix}
(a_1,a_2)(b_1,b_2)+(a_1,a_2)(c_1,c_2)
=(a_1+\tau_{a_2}(b_1)+a_1+\tau_{a_2}(c_1), a_2b_2+a_2c_2).
\end{matrix}
$$
Therefore, from the brace condition of $B$ we obtain an equality in the second component which tells us that we have a brace $B'$ of size $8$ with  additive group $E$ and multiplicative group $F$. Then, $B$ is the semidirect product via $\tau$ of the trivial brace with group $Z_p$ and this brace $B'$.
\end{proof}

In terms of Hopf-Galois structures this corresponds to
abelian types of induced structures as introduced in \cite{CRV}.

In the sequel, for $B$ a left brace of size $8p$ we shall denote by $N$ its additive group and by $G$ its multiplicative group. Then, $N=\Z_p\times E$, with $E$ an abelian group of order $8$, and $G=\Z_p\rtimes_{\tau} F$, with $F$ a group of order $8$ and $\tau:F\to\Aut(Z_p)$ a group homomorphism.

In order to classify the left braces of size $8p$ we can begin with the isomorphism classes of braces of size $8$ with additive group $E$ and then construct the semidirect products with $\Z_p$. Clearly, if we have isomorphic braces of size $8p$ we will have isomorphic braces of size $8$, but the converse is not true, since a brace of size $8$ can have different group morphisms $\tau:F\to\Aut(\Z_p)$ giving semidirect products which are non isomorphic braces.

Note that for $N=\mathbb{Z}_p\times E$, we have $\mathrm{Hol}(N)=\mathrm{Hol}(\mathbb{Z}_p)\times\mathrm{Hol}(E)$, and $G=\mathbb{Z}_p\rtimes_{\tau}F$ must be a subgroup of $\mathrm{Hol}(N)$, and in particular, $F$ is embedded in $\mathrm{Hol}(E)$. Now, in $\Hol(N)=\Hol(\Z_p)\times \Hol(E)$ we denote the elements $(m,k,a,\sigma)$ with $m,k$ integers mod $p$, $k\ne 0$, and
$(a,\sigma)\in E\rtimes\mathrm{Aut}(E)$. The element $(1,1,0,1)$ generates $\Z_p$ and, for $(a,\sigma)\in F$,
$$
(0,\tau(a,\sigma), a,\sigma)(1,1,0,1)(0,\tau(a,\sigma), a,\sigma)^{-1}=
$$
$$=
(\tau(a,\sigma),\tau(a,\sigma), a,\sigma)(0,\tau(a,\sigma)^{-1}, -\sigma^{-1}(a),\sigma^{-1})
=(\tau(a,\sigma),1,0,1)
=(1,1,0,1)^{\tau(a,\sigma)}.$$

Then, once fixed a homomorphism $\tau\colon F\longrightarrow\mathrm{Aut}(\Z_p)$, $$G=\{(m,\tau(a,\sigma),a,\sigma)\,|\,m\in\mathbb{Z}_p,\,(a,\sigma)\in F\}$$ is an order $8p$ group isomorphic to $\Z_p\rtimes_{\tau}F$. Since the action on $N$ is given by
$$(m,k,a,\sigma)(z,x)=(m+kz, a+\sigma(x))$$ we obtain a transitive action from transitivity in each component.

\begin{example}\label{ex}
Let $p$ be an odd prime and let $E=\Z_8$. Then, $\Hol(E)$ has a unique conjugacy class %length2
of regular subgroups isomorphic to $\Z_4\times \Z_2$.
Let $F=\langle f_4\rangle\times \langle f_2\rangle$ be one of them.
Then, we have two different group homomorphisms $\tau_1,\tau_2: F\to \Aut(\Z_p)=\Z_p^*$, with cyclic kernel of order $4$. These kernels are $\langle f_4\rangle$ and
$\langle f_4f_2\rangle$.

If we write $\Hol(E)=\Z_8\rtimes \Z_8^*$,
then we can take $f_4=(2,5)$ and $f_2=(1,7)$, since they have orders 4 and 2, respectively, they commute and $F=\langle f_4\rangle\times \langle f_2\rangle$ acts transitively on $\Z_8$ via $(a,l)x=a+lx$.

We have
$$
f_4f_2=(2,5)(1,7)=(2+5\cdot 1 \bmod 8,5\cdot 7\bmod 8)=(7,3).
$$
Since $\Z_8^*$ is abelian, conjugate elements share the same second component and we see that the cyclic subgroups $\langle f_4\rangle$ and $\langle f_4f_2\rangle$ are not conjugate in $\Hol(E)$.
%(in particular, not conjugate by an element of $\Aut(E)$).

\begin{comment}
Now, in $\Hol(\Z_p\times E)=\Hol(\Z_p)\times \Hol(E)$ we denote the elements $(m,k,a,l)$ with $m,k$ integers mod $p$, $k\ne 0$, and
$(a,l)\in \Z_8\rtimes \Z_8^*$. The element $(1,1,0,1)$ generates $\Z_p$ and for $i=1,2$
$$
(0,\tau_i(a,l), a,l)(1,1,0,1)(0,\tau_i(a,l), a,l)^{-1}=
$$
$$=
(\tau_i(a,l),\tau_i(a,l), a,l)(0,\tau_i(a,l)^{-1}, -l^{-1}a,l^{-1})
=(\tau_i(a,l),1,0,1)
=(1,1,0,1)^{\tau_i(a,l)}.$$
\end{comment}
For each $i\in\{1,2\}$
$$
G_i=\{(m,\tau_i(a,l), a,l)\mid\  m\in\Z_p,\  (a,l)\in F\}
$$
is a subgroup of $\Hol(N)$ isomorphic
to the semidirect product $\Z_p\rtimes_{\tau_i}F$. Since $G_1, G_2$ are regular subgroups of $\Hol(N)$, they correspond to two braces with addditive group $N$ and multiplicative group $G_1$ and $G_2$, respectively. To see that they are not isomorphic braces we have to check that $G_1$ and $G_2$ are not conjugate in $\Hol(N)$. We have
$$
\begin{matrix}G_1=\{
&(m, 1, 0, 1),&
(m, 1, 2, 5),& (m, 1, 4, 1),& (m, 1, 6, 5),\\
 &(m, -1, 1, 7),& (m, -1, 7, 3),& (m, -1, 5, 7),& (m, -1, 3, 3)
 \quad\}
\end{matrix}
$$
and
$$
\begin{matrix}G_2=\{
&(m, 1, 0, 1),&
(m, -1, 2, 5),& (m, 1, 4, 1),& (m, -1, 6, 5),\\
 &(m, -1, 1, 7),& (m, 1, 7, 3),& (m, -1, 5, 7),& (m, 1, 3, 3)\quad \}.
\end{matrix}
$$
Again, since $\Aut(\Z_p)$ and $\Aut(E)$ are abelian groups, conjugate elements in $\Hol(N)$ have the same values of the second and fourth components. Then, we see that $G_1$ and $G_2$ are not conjugate.

\end{example}

\section{Braces of order $8p$: direct products}

\begin{proposition}\label{direct} For an odd prime $p$, there are 27 left braces of size $8p$ which are direct product of the unique brace of size $p$ and a brace of size $8$.
\end{proposition}
\begin{proof}
In \cite{V} it is shown that there are 27 left braces of size 8. Then, the direct product of each of these with the trivial brace of size $p$ gives a left brace of size $8p$. %(and by Proposition \ref{braceprod} there is no one else).
\end{proof}

If we want to specify the multiplicative group of each brace above, we can use Magma to compute the conjugacy classes of regular groups of $\Hol(E)$ for the three different abelian groups of order $8$ and classify them according to the isomorphism class.
\begin{enumerate}
	\item 	$\Hol(\Z_8)\simeq \Z_8\rtimes V_4$ has 5 conjugacy classes of regular subgroups with the following distribution of isomorphism types
	\begin{center}
		\begin{tabular}{|l||c|} \hline
		Type  & Number \\
		\hline
		$\Z_8$&$2$\\
		\hline
		$\Z_4\times \Z_2$&$1$\\
		\hline
		$\Z_2\times \Z_2\times \Z_2$&$0$\\
		\hline
		$D_{2\cdot4}$&$1$\\
		\hline
		$Q_{8}$&$1$\\ \hline
		\end{tabular}
	\end{center}
	 This gives the number of braces with additive type $\Z_p\times \Z_8$ and multiplicative type a direct product
	 $\Z_p\times F$, with $F$ as in the above table.
	
	\item 	$\Hol(\Z_4\times \Z_2)\simeq (\Z_4\times \Z_2)\rtimes D_{2\cdot4}$ has 14 conjugacy classes of regular subgroups with the following distribution of isomorphism types
	\begin{center}
		\begin{tabular}{|l||c|} \hline
		Type  & Number \\
		\hline
		$\Z_8$&$0$\\
		\hline
		$\Z_4\times \Z_2$&$6$\\
		\hline
		$\Z_2\times \Z_2\times \Z_2$&$2$\\
		\hline
		$D_{2\cdot4}$&$5$\\
		\hline
		$Q_{8}$&$1$\\ \hline
		\end{tabular}
	\end{center}
	 This gives the number of braces with additive type $Z_p\times Z_4\times Z_2$ and multiplicative type a direct product $\Z_p\times F$, with $F$ as in the above table.

	\item 	$\Hol(\Z_2\times \Z_2\times \Z_2)\simeq \F_2^3\rtimes \GL(3,2) $ has $8$ conjugacy classes of regular subgroups with the following distribution of isomorphism types
	\begin{center}
		\begin{tabular}{|l||c|} \hline
		Type  & Number \\
		\hline
		$\Z_8$&$0$\\
		\hline
		$\Z_4\times \Z_2$&$3$\\
		\hline
		$\Z_2\times \Z_2\times \Z_2$&$2$\\
		\hline
		$D_{2\cdot4}$&$2$\\
		\hline
		$Q_{8}$&$1$\\ \hline
		\end{tabular}
	\end{center}
    This gives the number of braces with additive type $Z_p\times \Z_2\times \Z_2\times \Z_2$ and multiplicative type a direct product $\Z_p\times F$, with $F$ as in the above table.
	\end{enumerate}

\section{Braces of size $8p$: semidirect products}

\begin{proposition}\label{determsemid}
 Let $p=5$ or $p\ge 11$ be a prime and $N=\Z_p\times E$ an abelian group of order $8p$. % groups of order $8p$ are semidirect products $G=\Z_p\rtimes_{\tau} F$ with $F$ a group of order $8$ and $\tau:F\to\Aut(\Z_p)$ a group homomorphism.

 The conjugacy classes of regular subgroups of $\Hol(N)$ are in one to one correspondence with couples $(F,\tau)$ where $F$ runs over a set of representatives of conjugacy classes of regular subgroups of $\Hol(E)$ and $\tau$ runs over representatives of conjugacy classes by $\Aut(E)$ of group morphisms $\tau:F\to \Aut(\Z_p)$, that is
 $
 \tau\simeq \tau'$ if and only if $\tau=\tau'\circ \Phi_{\nu}|_F$
 where $\nu\in \Aut(E)$ and $\Phi_\nu$ is the corresponding inner automorphism of $\Hol(E)$.
\end{proposition}

\begin{proof} We know that groups of order $8p$ are semidirect products $G=\Z_p\rtimes_{\tau} F$ with $F$ a group of order $8$ and $\tau:F\to\Aut(\Z_p)$ a group homomorphism.

For a given couple $(F,\tau)$ the semidirect product is
$$
G=\Z_p\rtimes_{\tau} F=
\{(m,\tau(f), f)\mid\  m\in\Z_p,\  f\in F\}\subseteq
(\Z_p\rtimes \Z_p^*)\times \Hol(E)=\Hol(N)
$$
as in Example \ref{ex}. As we pointed out there,
the action on $N$ is given by
$(m,k,f)(z,x)=(m+kz, fx)$. $G$ containing $\Z_p$ gives transitivity in the first component and $G$ is regular in $\Hol(N)$ if and only if $F$ is regular in $\Hol(E)$.

Let us describe inner automorphisms of $\Hol(N)=(\Z_p\rtimes \Z_p^*)\times (E\rtimes \Aut(E))$. We write elements in
$\Hol(N)$ as $(m,k,a,\sigma)$ accordingly. Since we are dealing with regular subgroups, we just have to consider conjugation by elements
$(i,\nu)\in\Aut(N)=\Z_p^*\times \Aut(E)$. Let $\Phi_{(i,\nu)}$ be the inner automorphism of $(i,\nu)$ inside $\mathrm{Hol}(N)$. Then,
$$
\Phi_{(i,\nu)}(m,k,a,\sigma)=
(0,i,0,\nu)(m,k,a,\sigma)(0,i,0,\nu)^{-1}=
$$
$$
=(im,ik,\nu(a),\nu\sigma)(0,i^{-1},0,\nu^{-1})=
(im,k,\nu(a),\nu\sigma\nu^{-1})
$$
If we work in $\Hol(E)$, conjugation by $\nu\in\Aut(E)$ is
$$
\Phi_{\nu}(a,\sigma)=(0,\nu)(a,\sigma)(0,\nu^{-1})=(\nu(a),\nu\sigma\nu^{-1}).
$$
Let
$
G=\Z_p\rtimes_{\tau} F=
\{(m,\tau(a,\sigma), a,\sigma)\mid m\in\mathbb{Z}_p,\,(a,\sigma)\in F\}.
$
Then,
$$
\Phi_{(i,\nu)}(G)
=\{ (im, \tau(a,\sigma),\nu(a),\nu\sigma\nu^{-1})\mid m\in\mathbb{Z}_p,\,(a,\sigma)\in F\}.
$$
Since $i\in \Z_p^*$, $im$ runs over $\Z_p$ as $m$ does.
Therefore, if $(F',\tau')$ is another pair, we have
$$\Phi_{(i,\nu)}(G)=\Z_p\rtimes_{\tau'} F'\iff
F'=\Phi_{\nu}(F),\mbox{ and }
\tau=\tau'\circ \Phi_{\nu}|_F.
$$
Let us observe that in that case $\ker\tau'=\Phi_{\nu}(\ker\tau)$.
\end{proof}

\begin{remark}
The same result is valid for sizes $2^np$ with $p$ not dividing $2^n-1$, when all groups are semidirect products of the unique $p$-Sylow subgroup and a $2$-Sylow subgroup.
\end{remark}

In the previous section we have classified direct products, namely those cases with trivial morphism $\tau$. Now we are able to classify and count also proper semidirect products.

From section \ref{direct}  we know how many conjugacy classes of regular subgroups $\Hol(E)$ has and we have classified them according to their isomorphism types.
For each type we have to consider the possible morphisms $\tau$ and its conjugation class under $\Aut(E)$, as specified in Proposition \ref{determsemid}. From now on, the kernel of $\tau$ will be referred to as the kernel of the brace (or conjugation class of regular subgroups) determined by the pair $(F,\tau)$.

\subsection{$F \simeq \Z_8$}
This type only occurs with $E=\Z_8$ and we use the same notations of example \ref{ex}. Recall that $\Aut(E)=\Z_8^*$ and its nontrivial elements $l$ have order $2$.

 If $F$ is isomorphic to the cyclic group $\Z_8$ there is a unique morphism $\tau: F\to \Z_p^*$ with kernel of order $4$, the one sending generators to $-1$ and non-generators to $1$.

 If $p\equiv 1\bmod 4$, then $\Z_p^*$ has a (unique) subgroup of order $4$. Let $\zeta_4$ be a generator. Given a generator $(a,l)$ of $F$, we have two different morphisms with kernel of order $2$:  $\tau_1(a,l)=\zeta_4$ and $\tau_2(a,l)=\zeta_4^{-1}$. But then
 $\Phi_{-l}(a,l)=(-la,l)=(a,l)^{-1}$ and
 $\tau_1=\tau_2\circ\Phi_{-l}$. Every $\tau: F\to \Z_p^*$ with kernel of order $2$ is either $\tau_1$ or $\tau_2$ and therefore we have a unique pair $(F,\tau)$.

 If $p\equiv 1\bmod 8$, then $\Z_p^*$ has a (unique) subgroup of order $8$. Let $\zeta_8$ be a generator. Given a generator $(a,l)$ of $F$, we have $4$ different embeddings $F\to \Z_p^*$ given by $\tau_j(a,l)=\zeta_8^j$ for $j=1,3,5,7$.
 But then
 $$
 \tau_3=\tau_1\circ\Phi_{2+l}
 $$
 since $\Phi_{2+l}(a,l)=((2+l)a,l)=((1+l+l^2)a,l)=(a,l)^3$.
 Analogously, $\tau_5=\tau_1\circ\Phi_{3+2l}$ and
  $\tau_7=\tau_1\circ\Phi_{4+3l}$. Again, we have a unique pair $(F,\tau)$ for every $F$.

 \begin{proposition}\label{Z8}
  Let $p=5$ or $p\ge 11$ be a prime.
 \begin{enumerate}
 \item If $p\equiv 3,7\bmod 8$ there are $4$ left braces  with multiplicative group $\Z_p\rtimes \Z_8$. Two of them are direct products (kernel of order $8$) and the other two have kernel of order $4$.

 \item If $p\equiv 5\bmod 8$ there are $6$ left braces with multiplicative group $\Z_p\rtimes \Z_8$. Two of them are direct products, two of them have kernel of order $4$ and the other two have kernel of order $2$.

 \item If $p\equiv 1\bmod 8$ there are $8$ left braces with multiplicative group $\Z_p\rtimes Z_8$. Two of them are direct products, two of them have kernel of order $4$, two of them have kernel of order $2$ and the other two have trivial kernel.
 \end{enumerate}
 All the above braces have additive group $\Z_p\times\Z_8$.
 \end{proposition}

\subsection{$F \simeq \Z_4\times \Z_2$}\label{z4z2}

For $E=\Z_8$ and cyclic kernel of order $4$ it is the case of example \ref{ex}. We have just one $F$ and two non conjugate morphisms $\tau$.

On the other hand, there is a unique morphism $\tau: F\to \Z_p^*$ with kernel isomorphic to $\Z_2\times \Z_2$, it sends the elements of order $4$ to $-1$ and the other elements to $1$. For every $E$ we will have just as many semidirect products with elementary kernel as direct products.

If $p\equiv 1 \bmod 4$, then $\Z_p^*$ has a subgroup of order $4$. Let $\zeta_4$ be a generator. In this case we have morphisms $\tau$ with kernel of order $2$. Using the notation of example \ref{ex} the kernel can be either
$$
\langle f_2\rangle=\langle (1,7)\rangle
\ \mbox{ or }\
\langle f_4^2f_2\rangle=\langle (4,1)(1,7)\rangle=\langle (5,7)\rangle
$$
which are conjugate under $\Phi_5$. The four possible morphisms are
defined by
$$
\tau_1(2,5)=\tau_2(2,5)=\zeta_4,\quad \tau_1(1,7)= 1,\  \tau_2(1,7)= -1,
$$
$$
\tau_3(2,5)=\tau_4(2,5)=-\zeta_4,\quad \tau_3(1,7)= 1,\  \tau_4(1,7)= -1.
$$
Since $\Phi_5(2,5)=(2,5)$, we have $\tau_1=\tau_2\circ \Phi_5$ and
$\tau_3=\tau_4\circ \Phi_5$ while $\tau_1$ and $\tau_3$ are not conjugate.

\begin{proposition}
 Let $p=5$ or $p\ge 11$ be a prime.
 \begin{enumerate}
     \item If $p\equiv 3\bmod 4$ there are 4 left braces with additive group
     $\Z_p\times \Z_8$ and multiplicative group $\Z_p\rtimes (\Z_4\times \Z_2)$. One of them is a direct product, 2 of them have cyclic kernel of order 4 and the remaining one has kernel isomorphic to the Klein group.
     \item If $p\equiv 1\bmod 4$ there are 6 left braces with additive group
     $\Z_p\times \Z_8$ and multiplicative group $\Z_p\rtimes (\Z_4\times \Z_2)$. The distribution is as in 1 plus two braces with kernel of order 2.
 \end{enumerate}
\end{proposition}

Now we consider $E=\Z_4\times \Z_2$. We already know that there are 6 braces which are direct products and 6 which are semidirect products with kernel isomorphic to the Klein group.

The automorphism group of $\Z_4\times \Z_2$ is the dihedral group of order $8$. Using the classical notation of rotation and symmetry for its generators we have
$$
r(a,b)=(a+2b,a+b), \quad s(a,b)=(a,a+b) \quad \mbox{ for } (a,b)\in \Z_4\times \Z_2.
$$
%The automorphisms of order $2$ are $r^2, s,r^2s,rs,r^3s$
%and the automorphisms of order $4$ are $r, r^3$.
It is easy to check that
 $1+\sigma+\sigma^2+\sigma^3=0\in \End(E)$ for every $\sigma\in\Aut(E)$.
Therefore, we can write
$$\Hol(\Z_4\times \Z_2)=(\Z_4\times \Z_2)\rtimes D_{2\cdot 4}=\{ ((a,b), r^i s^{j})\mid a \bmod 4, b \bmod 2, 0\le i\le 3,\  j=0,1\}$$
and since $((a_1, b_1),\sigma_1)((a_2, b_2),\sigma_2)=((a_1, b_1)+\sigma_1(a_2,b_2),\sigma_1\sigma_2)$,
in $\Hol(E)$ all elements have order dividing $4$.

The conjugation by elements of $\Aut(E)$ is as follows
$$
((0,0), \nu)((a, b),\sigma)((0,0), \nu)^{-1}=
(\nu(a,b),\nu\sigma)((0,0), \nu^{-1})=(\nu(a,b),\nu\sigma\nu^{-1})
$$
so that we can work with conjugacy classes in $D_{2\cdot 4}$ and orbits under its action on $E$.
$((2,0), id)$ is invariant under conjugation since $(2,0)$ is fixed by $\Aut(E)$.

Since we are interested in regular subgroups we can rule out elements not acting with trivial stabilizers. The action is $((a,b),\sigma) (x,y)\to (a,b)+\sigma(x,y)$ and we have to rule out elements $((a,b),\sigma)$ such that $(a,b)$ is in the image of the endomorphism $1-\sigma$.

We have 6 conjugacy classes of elements of order 2 acting with trivial stabilizers
	\begin{center}
		\begin{tabular}{l|c}
	$\#$  &\\
\hline
		$1$&$((2,0),\ id)$ \\

	    $2$&$((0,1),\ id),\  ((2,1),\ id)$  \\

	    $2$&$((0,1),\ r^2),\ ((2,1),\ r^2)$\\

	    $4$&$((1,0),\ r^2),\ ((1,1),\ r^2)),\ ((3,0),\ r^2),\ ((3,1),\ r^2)$\\

	    $4$&$((2,0),\ s),\ ((2,0),\ r^2s),\ ((2,1),\ s),\ ((0,1),\ r^2)$\\

	    $4$&$((1,0),\ rs),\ ((1,1),\ r^3s),\ ((3,0),\ rs),\ ((3,1),\ r^3s)$\\

		\end{tabular}
	\end{center}
and  5 conjugacy classes of elements of order 4 acting with trivial stabilizers

\begin{center}
		\begin{tabular}{l|c}
	$\#$  &  \\
		\hline
		\hline
		$4$&$((1,0),\ id),\ ((1,1),\ id),\ ((3,0),\ id),\ ((3,1),\ id)$ \\
\hline
	    $4$&$((0,1),\ rs),\ ((2,1),\ rs),\ ((0,1),\ r^3s),\ ((2,1),\ r^3s)$  \\
\hline
	    $4$& $((1,1),\ rs),\ ((3,1),\ rs),\ ((1,0),\ r^3s),\ ((3,0),\ r^3s)$ \\
\hline
	    $8$&$((1,0),\ s),\ ((1,1),\ s),\ ((3,0),\ s),\ ((3,1),\ s),$\\
	
	    &$((1,0),\ r^2s),\ ((1,1),\ r^2s),\ ((3,0),\ r^2s),\ ((3,1),\ r^2s)$\\
\hline	
$8$&$((1,0),\ r),\ ((1,1),\ r),\ ((3,0),\ r),\ ((3,1),\ r)$\\
	    &$((1,0),\ r^3),\ ((1,1),\ r^3),\ ((3,0),\ r^3),\ ((3,1),\ r^3)$\\
\hline
		\end{tabular}
	\end{center}

From this we have 17 subgroups of order $2$ and  14 cyclic subgroups of order $4$.
Checking commutation of generators and conjugacy by $\Aut(E)$, we obtain the 6 conjugacy classes of
regular subgroups of $\Hol(E)$ we are looking for:
$$
\begin{array}{l}
F_1=\langle( (1,0),\ r)\rangle\times \langle (     (2,0),\ id)\rangle\\
F_2=\langle( (1,0),\ id)\rangle\times \langle (   (0,1),\  id)\rangle\\
F_3=\langle( (1,0),\ id)\rangle\times \langle (   (1,1),\  r^3s)\rangle\\
F_4=\langle( (1,0),\ s)\rangle\times \langle (    (2,0),\  id)\rangle\\
F_5=\langle( (0,1),\ rs)\rangle\times \langle (   (1,1),\  r^2)\rangle\\
F_6=\langle( (1,1),\ rs)\rangle\times \langle (   (0,1),\  r^2)\rangle\\
\end{array}.
$$

Now, for each $i=1,\dots,6$, we consider morphisms $\tau^{(i)}\colon F_i\to\Aut(\Z_p)$ and look for conjugate kernels.

In case of kernel of order 4, we proceed as in Example \ref{ex} with $f_4$ and $f_4f_2$. That is, if in the presentation of $F_i$ above we call $f_4$ the order $4$ element and $f_2$ the order $2$ one, we determine morphisms $\tau_1^{(i)}$, $\tau_2^{(i)}$ with kernels $\langle f_4\rangle$ and $\langle f_4f_2\rangle$ respectively and study their conjugation classes:
$$
\begin{array}{lll}
((1,0),r) & ((1,0),r)((2,0),id)=((3,0),r)& \mbox{ conjugated by }  \Phi_{r^2}\\
((1,0),id) & ((1,0),id)((0,1),id)=((1,1),id)& \mbox{ conjugated by } \Phi_{s}\\
((1,0),id) & ((1,0),id)((1,1),r^3s)=((2,1),r^3s)& \mbox{ not conjugated } \\
((1,0),s) & ((1,0),s)((2,0),id)=((3,0),s)& \mbox{ conjugated by } \Phi_{r^2}\\
((0,1),rs) & ((1,0),rs)((1,1),r^2)=((1,0),r^3s)& \mbox{ not conjugated } \\
((1,1),rs) & ((1,1),rs)((0,1),r^2)=((3,0),r^3s)& \mbox{ conjugated by } \Phi_{r^2s}\\
\end{array}
$$
Note that every conjugation $\Phi_{\nu}$ in the above table leaves the corresponding $F_i$ invariant. The first non-conjugacy class derives from non-conjugacy in $D_{2\cdot4}$ of $id$ and $r^3s$ while the second one derives from the non-existence of automorphisms carrying $(0,1)$ to $(1,0)$.
Since a kernel of order $4$ determines $\tau:F\to \Aut(\Z_p)$ we have that $F_3$ and $F_5$ provide two different semidirect products inside $\Hol(N)$ and each of the other $F_i$ provides just one.

If $p\equiv 1\bmod 4$ we can consider semidirect products with kernel of order $2$, and we proceed as before with possible kernels generated by $f_2$ and $f_4^2f_2$:
$$
\begin{array}{lll}
((2,0),id) & ((0,1),r^2)& \mbox{ not conjugated}  \\
((0,1),id) & ((2,1),id)& \mbox{ conjugated by  } \Phi_{r}
\\
((1,1),r^3s) & ((3,1),r^3s)& \mbox{ conjugated by  } \Phi_{r^2} \\
((2,0),id)&((0,1),id)& \mbox{ not conjugated } \\
((1,1),r^2)&((3,1),r^2)& \mbox{ conjugated by } \Phi_{r^2}\\
((0,1),r^2)&((2,1),r^2)& \mbox{ conjugated by } \Phi_{r}\\
\end{array}
$$
Both cases of non-conjugacy come from $\langle( (2,0),id)\rangle$ being normal.

Let us analyze the case of $F_2$, since the conjugation of kernels is not enough. The four possible group homomorphisms are
$$
\begin{array}{rccl}
     \tau_{\pm,\pm}:&\ F_2&\longrightarrow &\Z_p^*  \\
   &  ((0,1),id) & \to &\pm \zeta_4\\
  &   ((2,1),id)&  \to &\pm 1
\end{array}.
$$
%$$
%\tau_1((1,0),id)=\tau_2((1,0),id)=\zeta_4\quad \tau_1((0,1),id)= 1,\  %\tau_2((0,1),id)= -1
%$$
%$$
%\tau_3((1,0),id)=\tau_4((1,0),id)=-\zeta_4\quad \tau_3((0,1),id)= 1,\  \tau_4((0,1),id)=- 1
%$$
We have conjugations
$$
\Phi_{r^2}:\begin{array}{l}
((1,0),id)\to ((1,0),id)^3=((3,0),id)\\
((0,1),id)\to (((0,1),id)\\
\end{array}
$$
$$
\Phi_{r^3s}:\begin{array}{l}
((1,0),id)\to ((1,0),id)\\
((0,1),id)\to (((2,1),id)\\
\end{array}
\qquad
\Phi_{rs}:\begin{array}{l}
((1,0),id)\to (((3,0),id)\\
((0,1),id)\to (((2,1),id)\\
\end{array}
$$
which give
$$
\tau_{-+}=\tau_{++}\circ\Phi_{r^2}, \quad
\tau_{+-}=\tau_{++}\circ\Phi_{r^3s}, \quad
\tau_{--}=\tau_{++}\circ\Phi_{rs}.
$$
Therefore, $F_2$ provides a unique conjugacy class.
In the following table we give the conjugations for all cases:
\begin{center}
		\begin{tabular}{l|ccc}
		\\
		\hline
		$F_2$& $\Phi_{r^2}$ $\Phi_{r^3s}$ $\Phi_{rs}$ \\
		$F_3$& $\Phi_{rs}$ $\Phi_{r^3s}$ $\Phi_{r^2}$ \\
		$F_5$& $\Phi_{rs}$ $\Phi_{r^2}$ $\Phi_{r^3s}$ \\
		$F_6$& $\Phi_{r^2}$ $\Phi_{rs}$ $\Phi_{r^3s}$ \\
		\end{tabular}
\end{center}
Therefore, each one of these groups provides exactly one conjugacy class.
For $F_1$ and $F_4$ a generator of order $2$ is invariant under conjugation. Since we have $\Phi_{s}((1,0),r))=((1,1),r^3)=((1,0),r)^3$ and
$\Phi_{r^2s}((1,0),s))=((3,1),s)=((1,0),s))^3$,
 each group provides exactly two conjugacy classes.

\begin{proposition}
 Let $p=5$ or $p\ge 11$ be a prime.
 \begin{enumerate}
     \item If $p\equiv 3\bmod 4$ there are 20 left braces with additive group
     $\Z_p\times \Z_4\times \Z_2$ and multiplicative group $\Z_p\rtimes (\Z_4\times \Z_2)$. Six of them are direct products, 8 of them have cyclic kernel of order $4$ and the remaining 6 have kernel isomorphic to the Klein group.
     \item If $p\equiv 1\bmod 4$ there are 28 left braces with additive group
     $\Z_p\times  \Z_4\times \Z_2$ and multiplicative group $\Z_p\rtimes (\Z_4\times \Z_2)$. The distribution is as in 1 plus 8 braces with kernel of order $2$.
 \end{enumerate}
\end{proposition}

The last additive type is $E=\Z_2\times \Z_2\times \Z_2$. We already know that there are 3 braces which are direct products and 3 which are semidirect products with kernel isomorphic to the Klein group.

Since we can identify the additive group with the binary vector space of dimension $3$, its automorphism group is the group of $3\times 3$ invertible binary matrices and $\Hol(\Z_2\times \Z_2\times \Z_2)\simeq \F_2^3\rtimes \GL(3,2) $.
Therefore, we can write
$$\Hol(\Z_2\times \Z_2\times \Z_2)=\{ (v, M) :\quad v\in\F_2^3,\ M\in\GL(3,2)\}.$$
The operation is given by $
(v_1, M_1)(v_2, M_2)=(v_1+M_1v_2, M_1M_2)
$
and the action on $\F_2^3$ by $(v,M) u = v+Mu.$ In order to act with trivial stabilizers we need $v\not\in Im(M+Id).$

$\GL(3,2)$ is a simple group of order 168 which has
a unique conjugacy class of elements of order $2$, of length 21,
with representative
$$
S=    \begin{bmatrix}
1&0&0\\
0&1&1\\
0&0&1\\
    \end{bmatrix}
    $$
and a unique conjugacy class of elements of order $4$, of length 42,
with representative
$$
 Q=   \begin{bmatrix}
1&1&0\\
0&1&1\\
0&0&1\\
    \end{bmatrix}.
    $$
For $S$ we have rank$(S+Id)=1$ and $Im(S+Id)\subset \Ker(S+Id)$. For $Q$ we have rank$(Q+Id)=2$ and $Id+Q+Q^2+Q^3=0$.

The elements of order 2 in $\Hol(E)$ distribute in $3$ conjugacy classes of lengths 7, 42 and 42, respectively, but only two of them correspond to elements acting with trivial stabilizers. Since
$(v, M)^2=((M+Id)v, M^2)$,
 the element $(v,M)$ has order 2 if and only if either $M=Id$ and $v\ne 0$ or $M$ has order $2$ and $v=0$ or $v$ is eigenvector of eigenvalue $1$. Therefore, the elements of order $2$ acting with trivial stablilizers are
$$
(u,\ Id),\ (v_1,\ M),\ (v_2,\ M)
$$
$u\ne 0$, $M$ of order 2 and $v_1,v_2\in \ker(M+Id)$, $v_1,v_2\not\in Im(M+Id)$.

The elements of order 4 in $\Hol(N)$ distribute in 3 conjugacy classes of lengths 84, 168, 168, respectively. Again, only two of them correspond to actions with trivial stabilizers. %One of length 168 non trivial stabilizers
Since $(v, M)^4=((M^3+M^2+M+Id)v, M^4)$, we can have
$M$ of order $2$ and $v$ one of the 4 vectors not in $\ker(M+Id)$ or
$M$ of order $4$ and any $v$, since $M^3+M^2+M+Id=0$.  Now, $M+Id$ has rank $2$ and we have 4 vectors in $Im(M+Id)$.

Let us now look for the three conjugacy classes of subgroups of $\Hol(E)$ isomorphic to $\Z_4\times \Z_2$. Let us use the notation $e_1,e_2,e_3$ for the canonical basis of $\F_2^3$. Since $e_3$ is not an eigenvector of $S$, the element $(e_3,S)$ has order $4$ in $\Hol(E)$.  Let us look for elements of order 2 commuting with it and different from $(e_3,S)^2=(e_2,Id)$.
For elements of type $(u, Id)$ we have
$$
(e_3,S)(u,Id)=(u,Id)(e_3,S)\iff e_3+Su=u+e_3\iff u\in \ker(S+Id).
$$
We can choose $u=e_1$ or $u=e_1+e_2$ but both give the same regular subgroup
$$\begin{array}{rcl}
F_1&=&\langle (e_3, S)\rangle \times \langle (e_1, Id)\rangle=\\
&=&\{(0,Id),(e_3, S),(e_2, Id),(e_2+e_3, S),\\
&&\ (e_1, Id),(e_1+e_3, S),(e_1+e_2, Id),(e_1+e_2+e_3, S)\}
\end{array}
$$
with pairs of non-eigenvectors with $S$ and eigenvectors with $Id$.
%$(e_3,S)$ and $(e_1+e_3,S)$ conjugated by some $\Phi_M$? We need $Me_3=e_1+e_3$ and $MSM^{-1}=S$

For elements of order 2 of type $(v, M)$ we have
$$
(e_3,S)(v,M)=(v,M)(e_3,S)\iff e_3+Sv=v+Me_3 \mbox{ and } MS=SM
$$
Note that we cannot take $M=S$ because $v$ is an eigenvector of $M$ and $e_3$ is not an eigenvector of $S$.
We need elements of order $2$ in the centralizer of $S$ in $\GL(3,2)$, which is a dihedral group of order $8$.
We take the unique possible matrix
$$M=\begin{bmatrix}
1&0&0\\
1&1&1\\
0&0&1\\
    \end{bmatrix}
$$
and $v=e_1+e_3$, which is in the kernel of $M+Id$ but not in the image.
Then, $e_3+S(e_1+e_3)=e_1+e_2=e_1+e_3+Me_3$ and
we obtain a regular subgroup
$$\begin{array}{rcl}
F_2&=&\langle (e_3, S)\rangle \times \langle (e_1+e_3, M)\rangle=\\
&=&\{(0,Id),(e_3, S),(e_2, Id),(e_2+e_3, S),\\
&&\ (e_1+e_3, M),(e_1+e_2, MS),(e_1+e_2+e_3, M),(e_1, MS)\}.
\end{array}
$$
Taking the other eigenvector $e_1+e_2+e_3$ we obtain the same subgroup.
%$(e_3,S)$ and $(e_1+e_2,MS)$ conjugated by some $\Phi_A$? We need $Ae_3=e_1+e_2$ and $ASA^{-1}=MS$

Now we take the element $(e_3,\ Q)$ of order 4 and search for elements of order 2 commuting with it. If it is of type
$(u,\ Id)$ we need $u+e_3=e_3+Qu$ and we should take the unique non-zero eigenvector of $Q$, which is $e_1$.
We obtain a regular subgroup
$$\begin{array}{rcl}
F_3&=&\langle (e_3, Q)\rangle \times \langle (e_1, Id)\rangle=\\
&=&\{(0,Id),(e_3, Q),(e_2, Q^2),(e_1+e_2+e_3, Q^3),\\
&&\ (e_1, Id),(e_1+e_3, Q),(e_1+e_2, Q^2),(e_2+e_3, Q^3)\}
\end{array}
$$
Let us remark that the centralizer of $Q$ in $\GL(3,2)$ is the subgroup generated by $Q$ and therefore there are no elements of order $2$ commuting with $Q$ except for $Q^2$.
%$(e_3,Q)$ and $(e_1+e_3,Q)$ conjugated by some $\Phi_A$? We need $Ae_3=e_1+e_3$ and $AQA^{-1}=Q$ --> A= Q^2

The next step is once again to consider morphisms $\tau_i:F_i\to \Aut(\Z_p)$ and check for conjugate kernels.
Recall that conjugation by an element of $\Aut(E)$ is $\Phi_D(v,A)=(Dv, DAD^{-1}).$

In case of kernel of order 4 we have, respectively,
$$
\begin{array}{lll}
(e_3,S) &  (e_1+e_3,S)& \mbox{ conjugated by }  \Phi_{M'}\\
(e_3,S) & (e_1+e_2,MS)& \mbox{ conjugated by } \Phi_{\tilde M}\\
(e_3,Q)&(e_1+e_3,Q)& \mbox{ conjugated by } \Phi_{Q^2}\\
\end{array}
$$
where
$$M'=\begin{bmatrix}
1&0&1\\
0&1&0\\
0&0&1\\
    \end{bmatrix}
$$
is in the centralizer of $S$ and
$$\tilde M=\begin{bmatrix}
0&0&1\\
1&1&1\\
1&0&0\\
    \end{bmatrix}
$$
is in the centralizer of $M$ and such that $\tilde M\, S\, \tilde M^{-1}=MS$.
Therefore, every $F_i$ provides a unique conjugacy class.

If $p\equiv 1\bmod 4$ we can consider semidirect products with kernel of order $2$, and we proceed as before with possible kernels:
$$
\begin{array}{lll}
(e_1,Id) &  (e_1+e_2,Id)& \mbox{ conjugated by }  \Phi_{M}\\
(e_1+e_3,M) & (e_1+e_2+e_3,M)& \mbox{ conjugated by } \Phi_{MS}\\
(e_1,Id)&(e_1+e_2,Q^2)& \mbox{ not conjugate } \\
\end{array}
$$
Note that $MS\in\Cent_{\GL(3,2)}(M)\cap \Cent_{\GL(3,2)}(S)$. %which is isomorphic to the Klein group.

For $F_1$ the four possible group homomorphisms are
$$
\begin{array}{rccl}
     \tau_{\pm,\pm}:&\ F_2&\longrightarrow &\Z_p^*  \\
   &  (e_3,S) & \to &\pm \zeta_4\\
  &   (e_1,id)&  \to &\pm 1
\end{array}.
$$
We have
$$
\tau_{-+}=\tau_{++}\circ\Phi_{S}, \quad
\tau_{+-}=\tau_{++}\circ\Phi_{MS}, \quad
\tau_{--}=\tau_{++}\circ\Phi_{M}.
$$
and $F_1$ provides a unique conjugacy class. For $F_2$ the four possible group homomorphisms are
$$
\begin{array}{rccl}
     \tau_{\pm,\pm}:&\ F_2&\longrightarrow &\Z_p^*  \\
   &  (e_3,Q) & \to &\pm \zeta_4\\
  &   (e_1,Id)&  \to &\pm 1
\end{array}.
$$
We have
$$
\tau_{-+}=\tau_{++}\circ\Phi_{M}, \quad
\tau_{+-}=\tau_{++}\circ\Phi_{MS}, \quad
\tau_{--}=\tau_{++}\circ\Phi_{S}.
$$
and $F_2$ provides also a unique conjugacy class. For $F_3$, since $\Phi_{S}$ leaves $(e_1,Id)$ invariant and takes $(e_3, Q)$ to $(e_2+e_3, Q^3)$ we have
$\tau_{-+}=\tau_{++}\circ\Phi_{S}$ and therefore $F_3$ provides two different conjugacy classes of semidirect products with kernel of order $2$.

\begin{proposition}
 Let $p=5$ or $p\ge 11$ be a prime.
 \begin{enumerate}
     \item If $p\equiv 3\bmod 4$ there are 9 left braces with additive group
     $\Z_p\times \Z_2\times \Z_2\times \Z_2$ and multiplicative group $\Z_p\rtimes (\Z_4\times \Z_2)$. Three of them are direct products, 3 of them have cyclic kernel of order $4$ and the remaining 3 have kernel isomorphic to the Klein group.
     \item If $p\equiv 1\bmod 4$ there are 13 left braces with additive group
     $\Z_p\times  \Z_2\times \Z_2\times \Z_2$ and multiplicative group $\Z_p\rtimes (\Z_4\times \Z_2)$. The distribution is as in case 1 plus 4 braces with kernel of order $2$.
 \end{enumerate}
\end{proposition}

\subsection{$F \simeq \Z_2\times \Z_2\times \Z_2$}\label{z2z2z2}

This case only occurs when the abelian group is $\Z_4\times \Z_2$ or $\Z_2\times \Z_2\times \Z_2$

When $E=\Z_4\times \Z_2$ in $\Hol(E)$ there are two conjugacy classes of regular subgroups isomorphic to $\Z_2\times \Z_2\times \Z_2$.
They are normal, therefore union of conjugacy classes, and they intersect in the normal subgroup of order $2$.
Working with the conjugacy classes of elements of order $2$ described in the previous subsection we find
$$
\begin{array}{rl}
     F_1&=\langle ((2,0),id)\rangle\times \langle ((1,0), r^2)\rangle\times\langle ((1,1), r^2) \rangle \\
     F_2&=\langle ((2,0),id)\rangle\times \langle ((0,1), r^2)\rangle\times\langle ((1,0), rs) \rangle \\
\end{array}
$$
and we have to count classes of morphisms $\tau^{(i)}: F_i\to \Z_p^*$ with kernel of order 4, therefore isomorphic to the Klein group. We can freely choose two elements from the nontrivial ones and in this way we obtain 7 possible kernels and each element belongs to three different subgroups.
Since $((2,0),id)$ is invariant under conjugation, kernels containing this element cannot be conjugate to kernels not containing it. Let us see if they give a single conjugacy class.

For $F_1$ the three kernels containing $((2,0),id)$ are
$$
\begin{array}{cc}
\langle ((1,0),r^2)\, ,  & ((3,0),r^2)\rangle \\
\langle((1,1),r^2)\, ,   & ((3,1),r^2)\rangle \\
\langle((0,1),id)\, , & ((2,1),id)\rangle.
\end{array}
$$
Conjugation $\Phi_r$ takes the first to the second one. But these two groups are not conjugated to the third one.
The four kernels not containing $((2,0),id)$ are
$$
\begin{array}{cc}
\langle ((1,0),r^2)\, ,  & ((0,1),id)\rangle \\
\langle((1,0),r^2)\, ,   & ((2,1),id)\rangle \\
\langle((3,0),r^2)\, , & ((2,1),id)\rangle\\
\langle((3,0),r^2)\, , & ((0,1),id)\rangle\\
\end{array}
$$
The automorphism $r^3s$ has fixed points $(1,0)$ and $(3,0)$ and exchanges $(0,1)$ and $(2,1)$, therefore
$\Phi_{r^3s}$ gives conjugacy of the first with the second and the third with the fourth one.
Analogously we see that the first and third kernels are conjugate by $\Phi_{rs}$.
All together we obtain three conjugacy classes from $F_1$.

For $F_2$ the three kernels containing $((2,0),id)$ are
$$
\begin{array}{cc}
\langle ((0,1),r^2)\, ,  & ((2,1),r^2)\rangle \\
\langle((1,0),rs)\, ,   & ((3,0),rs)\rangle \\
\langle((3,1),r^3s)\, , & ((1,1),r^3s)\rangle.
\end{array}
$$
The first one cannot be conjugate to the other two because elements of the second component are not conjugate in $D_{2\cdot 4}$. The conjugation $\Phi_{r^2s}$ takes the second to the third one.
The four kernels not containing $((2,0),id)$ are
$$
\begin{array}{cc}
\langle ((0,1),r^2)\, ,  & ((3,0),rs)\rangle \\
\langle((0,1),r^2)\, ,   & ((1,1),r^3s)\rangle \\
\langle((2,1),r^2)\, , & ((1,0),rs)\rangle\\
\langle((2,1),r^2)\, , & ((3,0),rs)\rangle\\
\end{array}
$$
We see that
$\Phi_{r^2s}$ gives conjugacy of the first and the second one, $\Phi_{r^2}$ gives conjugacy of the third and the fourth one, and $\Phi_{rs}$ gives conjugacy of the first and the third one.
All together we obtain three conjugacy classes from $F_2$.

\begin{proposition}
 Let $p=5$ or $p\ge 11$ be a prime. There are 8 left braces with additive group
     $\Z_p\times \Z_4\times \Z_2$ and multiplicative group $\Z_p\rtimes (\Z_2\times \Z_2\times \Z_2)$. Two of them are direct products and the remaining 6 have kernel isomorphic to the Klein group.
\end{proposition}

When $E=\Z_2\times \Z_2\times \Z_2$ in $\Hol(E)\simeq \F_2^3\times \GL(3,2)$ there are also
two conjugacy classes of regular subgroups isomorphic to $\Z_2\times \Z_2\times \Z_2$.

One of them has length 1 and comes from the conjugacy class of elements of order $2$ with identity matrix in the second component, namely from the natural embedding of $\Z_2\times \Z_2\times \Z_2$ in its holomorph:
$$
F_1=\{(v,Id): v\in \F_2^3\}
$$
In order to generate a second one we need elements $(u, Id)$, $(v, S)$, $(w, A)$ such that
$$
u+v=v+Su,\quad u+w=w+Au,\quad v+Sw=w+Av,\quad A^2=\Id, \quad AS=SA
$$
with $v\in\Ker(S+Id)\setminus Im(S+Id)$ and $w\in\Ker(A+Id)\setminus Im(A+Id)$. Therefore, $u\ne 0$ is a common eigenvector of $S$ and $A$. If $A=\Id$, we should have $4$ different elements of order $2$ with $S$ in the second component, but there are only $2$.
Therefore $A$ should be in the centralizer of $S$ and have a common eigenvector $u\ne 0$ with $S$, that is, $u\in \langle e_1,e_2\rangle $. Finally, the condition $v+Sw=w+Av$ implies that $Im(S+Id)=Im(A+Id)$.

Let us take $A=M$, as in previous subsection, and $u=e_2$. Then, $v=e_1$ is a valid eigenvector of $S$ and $w=e_1+e_3$ is a valid eigenvector of $M$. We have $v+Sw=e_1+e_1+e_2+e_3=e_2+e_3$ and $w+Mv=e_1+e_3+e_1+e_2=e_2+e_3$. Therefore, we have the second conjugacy class of regular elementary subgroups of $\Hol(E)$:
$$\begin{array}{rcl}
F_2&=&\langle (e_2, Id)\rangle \times \langle (e_1, S)\rangle \times\langle (e_1+e_3, M)\rangle =\\
&=&\{(0,Id),(e_2, Id),(e_1, S),(e_1+e_3, M),\\
&&\ (e_1+e_2, S),(e_1+e_2+e_3, M),(e_2+e_3, SM),(e_3, SM)\}.
\end{array}
$$

Again, there are 7 possible kernels of order $4$ for every $F_i$. For $F_1$, the first components form a 2-dimensional vector subspace of $\F_2^3$ and $\GL(3,2)$ acts transitively on this set of subspaces. Therefore, any two of them are conjugated by some $\Phi_D$, with $D\in\GL(3,2)$. All these conjugations $\Phi_D$ leave $F_1$ invariant and this subgroup provides a unique conjugacy class of semidirect products.

Let us analyze the classes of Klein subgroups of $F_2$. Three of them contain the element $(e_2, Id)$, which has to be invariant under any conjugation $\Phi_D:F_2\to F_2$. Therefore, they cannot be conjugated to any of the other four subgroups. Let us see that they form a conjugacy class. These kernels are
$$
\begin{array}{rcl}
K_1&=&\langle (e_2,Id)\, ,   (e_1,S)\rangle \\
K_2&=&\langle (e_2,Id)\, ,    (e_1+e_3,M)\rangle \\
K_3&=&\langle (e_2,Id)\, ,   (e_2+e_3,SM)\rangle.
\end{array}
$$
Keeping the above notation, $\Phi_{\tilde M}$ leaves $F_2$ invariant and $\Phi_{\tilde M}(K_1)=K_3$.
Taking
$$\tilde{\tilde M}=\begin{bmatrix}
1&0&0\\
0&1&0\\
1&0&1\\
    \end{bmatrix}
$$
we have that $\Phi_{\tilde{\tilde M}}$ leaves $F_2$ invariant and $\Phi_{\tilde{\tilde M}}(K_1)=K_2$.

The remaining kernels are
$$
\begin{array}{rcl}
K_4&=&\langle (e_1,S)\, ,   (e_1+e_3,M)\rangle \\
K_5&=&\langle (e_1,S)\, ,   (e_1+e_2+e_3,M)\rangle \\
K_6&=&\langle (e_1+e_2,S)\, ,   (e_1+e_3,M)\rangle\\
K_7&=&\langle (e_1+e_2,S)\, ,   (e_1+e_2+e_3,M)\rangle.
\end{array}
$$
$MS\in\Cent_{\GL(3,2)}(M)\cap \Cent_{\GL(3,2)}(S)$ gives $\Phi_{MS}(K_4)=K_7$ and $\Phi_{MS}(K_5)=K_6$.
On the other hand, $\Phi_{S}(K_4)=K_5$. %and $\Phi_{M}(K_4)=K_6$

\begin{proposition}
 Let $p=5$ or $p\ge 11$ be a prime. There are 5 left braces with additive group
     $\Z_p\times \Z_2\times \Z_2\times \Z_2$ and multiplicative group $\Z_p\rtimes (\Z_2\times \Z_2\times \Z_2)$. Two of them are direct products and the other 3 have kernel isomorphic to the Klein group.
\end{proposition}

\subsection{$F \simeq D_{2\cdot 4}$}\label{d24}

Let us determine dihedral subgroups of the different holomorphs. Let us observe that in this case we never have kernels of order $2$ since the unique normal subgroup of order $2$ is generated by the square of an element of order $4$. Therefore, we should consider cyclic kernels and Klein kernels, and then conjugacy of kernels by automorphisms is the unique condition  we need to classify semidirect products. Since a group $F\simeq D_{2\cdot 4}$ has a unique cyclic subgroup of order $4$, for every possible $F$ there is just one semidirect product with cyclic kernel.

When $E=\Z_8$ in $\Hol(E)$ we have just one regular dihedral subgroup, which is normal and therefore union of conjugacy classes. Since there is just one conjugacy class of elements of order 2 acting with trivial stabilizers and, as we have seen, its elements commute with $(2,5)$, we have to take the other conjugacy class of order $4$ and length $2$.
We check
$$
(1,7)(2,1)(1,7)=(6,1)=(2,1)^3
$$
so that
$
F=\langle (2,1), (1,7)  \rangle.
$
We have two Klein kernels,$
 \langle (1,7), (5,7) \rangle$ and
$\langle (3,7), (7,7) \rangle,     $
which are conjugate by $\Phi_3$.

\begin{proposition}
 Let $p=5$ or $p\ge 11$ be a prime. There are $3$ left braces with additive group
     $\Z_p\times \Z_8$ and multiplicative group $\Z_p\rtimes D_{2\cdot 4}$.
     One is a direct product, another one has cyclic kernel of order 4 and the third one has kernel isomorphic to the Klein group.
\end{proposition}

Next we consider $E=\Z_4\times \Z_2$ in whose holomorph we have 5 conjugacy classes of regular subgroups isomorphic to  $D_{2\cdot 4}$. We start with the five conjugacy classes of cyclic subgroups of order 4 obtained in subsection \ref{z4z2}.
$$
\langle( (1,0),\ r)\rangle,\
\langle( (1,0),\ id)\rangle,\
\langle( (1,0),\ s)\rangle,\
\langle( (0,1),\ rs)\rangle,\
\langle( (1,1),\ rs)\rangle.
$$
Then, for each of the subgroups $\langle f_4\rangle$ we have to consider
elements $f_2$ of order $2$ such that $f_2f_4f_2=f_4^{-1}$.
We find
$$
\begin{array}{l}
((2,0),s) ((1,0), r)((2,0),s)=((3,1),r^3s)((2,0),s)=((1,1),r^3)=((1,0),r)^3    \\[2ex]

((0,1),r^2) ((1,0), id)((0,1),r^2)=  ((3,1), r^2)((0,1),r^2) =((3,0),id)=((1,0),id)^3  \\[2ex]

((1,1),r^2) ((1,0), s)((1,1),r^2)=  ((0,1), r^2s)((1,1),r^2) =((3,1),s)=((1,0),s)^3  \\[2ex]

((1,0),r^2) ((0,1), rs)((1,0),r^2)=  ((1,1), r^3s)((1,0),r^2) =((2,1),rs)=((0,1),rs)^3  \\[2ex]

((2,1),id) ((1,1), rs)((2,1),id)=  ((3,0), rs)((2,1),id) =((3,1),rs)=((1,1),rs)^3  \\[2ex]
\end{array}
$$

Each of the corresponding regular dihedral groups $F_i$ provides two possible Klein kernels: $\langle f_4^2, f_2\rangle$ and $\langle f_4^2, f_4f_2\rangle$.

\begin{center}
\begin{tabular}{r|l|l|l}
& $K_1$ & $K_2$ & \\
\hline
   $F_1$ & $((2,1),r^2)\ ((2,0),s)$   &$((2,1),r^2)\ ((3,0),rs) $& Not conjugate  \\
  $F_2$ & $((2,0),id)\ ((0,1),r^2)$   &$((2,0),id)\ ((1,1),r^2) $& Not conjugate  \\
  $F_3$ & $((2,1),id)\ ((1,1),r^2)$   &$((2,1),id)\ ((2,0),r^2s) $& Not conjugate  \\
  $F_4$ & $((2,0),id)\ ((1,0),r^2)$   &$((2,0),id)\ ((3,1),r^3s) $& Not conjugate \\
  $F_5$ & $((2,0),id)\ ((2,1,id)$   &$((2,0),id)\ ((1,0),rs) $& Not conjugate \\
\end{tabular}
\end{center}
Therefore, every $F_i$ provides two non-conjugate semidirect products.

\begin{proposition}
 Let $p=5$ or $p\ge 11$ be a prime. There are $20$ left braces with additive group
     $\Z_p\times \Z_4\times\Z_2$ and multiplicative group $\Z_p\rtimes D_{2\cdot 4}$.
    Five are direct products, five have cyclic kernel of order 4 and the remaining ten have kernel isomorphic to the Klein group.
\end{proposition}

In $\Hol(\Z_2\times\Z_2\times\Z_2)$ there are two conjugacy classes of regular subgroups isomorphic to $D_{2\cdot 4}$.  A representative for one of the conjugacy classes of elements of order $4$ is $(e_3,S)$.
Its square  is $(e_2,Id)$ and its cube $(e_2+e_3,S)$.
If we consider an element of order $2$ of the form $(u,Id)$
$$
(e_3,S)(u,Id)=(u,Id)(e_2+e_3,S)\iff e_3+Su=u+e_2+e_3\iff Su=u+e_2
$$
We take $u=e_1+e_2+e_3$ and the dihedral group is
$$
F_1=\{(0,Id), (e_3,S),(e_2,Id), (e_2+e_3,S), (e_1+e_2+e_3,Id), (e_1,S), (e_1+e_3,Id), (e_1+e_2,S)\}
$$
The two Klein kernels are $\langle(e_2,Id),(e_1+e_2+e_3,Id))\rangle$ and $\langle(e_2,Id),(e_1,S)\rangle$ and they are not conjugate.

Since $SQ=Q^3S$ we can consider an element of order $2$ with matrix $S$.
In this way we obtain
$$
(e_1+e_2,S)(e_3,Q)=(e_1+e_3,SQ)\quad (e_1+e_2+e_3,Q^3)(e_1+e_2,S)=(e_1+e_3,Q^3S)
$$
and the second regular group
$$
\begin{array}{lll}
F_2&=&\langle (e_3,Q),(e_1+e_2,S)\rangle\\
&=&\{(0,Id), (e_3,Q),(e_2,Q^2), (e_1+e_2+e_3,Q^3),\\
&&(e_1+e_2,S),(e_1+e_3,SQ) ,(e_1,SQ^2), (e_2+e_3,SQ^3)\}.\end{array}
$$
The two Klein kernels are
$K_1=\langle(e_2,Q^2),(e_1+e_2,S)\rangle$
and $K_2=\langle(e_2,Q^2),(e_1+e_3,SQ)\rangle$. They are not conjugate since the vectors in the elements of $K_1$ form the subspace $\langle e_1, e_2\rangle$ but the vectors in the second one do not form a subspace, therefore we cannot have a matrix carrying the first set of vectors into the second one.

\begin{proposition}
 Let $p=5$ or $p\ge 11$ be a prime. There are $8$ left braces with additive group
     $\Z_p\times \Z_2\times\Z_2\times\Z_2$ and multiplicative group $\Z_p\rtimes D_{2\cdot 4}$.
    Two are direct products, two have cyclic kernel of order 4 and the  remaining 4 have kernel isomorphic to the Klein group.
\end{proposition}

%%
%
%SQ=[1 1 0]
%[0 1 0]
%[0 0 1]=Q^3S
%%

%Q^2=[1 0 1]
%[0 1 0]
%[0 0 1]

%Q^3=[1 1 1]
%[0 1 1]
%[0 0 1]

%sQ^2=[1 0 1]
%[0 1 1]
%[0 0 1]

%If we consider an element of order $2$ of the form $(v,A)$, then $A\ne S$ has to be an element of order $2$ in the centralizer of $S$. The possibilities are $A=M,M',SM$ or $SM'$, with the notations of the previous subsections.
%$$
%((e_3,S)(u,A)=(u,A)(e_2+e_3,S)\iff e_3+Su=u+A(e_2+e_3)\iff Su+u=e_3+A(e_2+e_3)
%$$
%Since $Su+u$ is in the image of $S+Id$ it can be $0$ or $e_2$, therefore, $A(e_2+e_3)$ can be
%$e_3$ or $e_2+e_3$. This gives $A=M$ or $A=MS$. If $A=M$ we get $Su=u$ and $u$ has to be a common eigenvalue for $M$ and $S$ but this is just $Im(M+Id)$. If $A=MS$ the possible eigenvalues are $u=e_3,e_2+e_3$ and $Su+u=e_2$ gives $u=e_3$. But then
%$$
%(e_3,S)(e_3,MS)=(e_3+Se_3,SMS)=(e_2,M).
%$$
%Since $e_2$ belongs to $Im(M+Id)$ this is not a valid element to form a regular subgroup.

%To get the second regular subgroup we have to consider matrices of order $4$ in $GL(3,2)$.

%We can take matrices in the centralizer of $S$, which is a dihedral group.
%Let $$Q'=\begin{bmatrix}
%1&0&1\\
%1&1&1\\
%0&0&1\\
%    \end{bmatrix}.
%$$
%If $(e_1+e_3,Q')$ is an element of order 4 operating with trivial stabilizers,
%$(e_1+e_3,Q')^2=(e_1,S)$ and $(e_1+e_3,Q')^3=(e_2+e_3,SQ')$. Then,
%$$
%(v,M)(e_1+e_3,Q')=(v+e_1+e_3, MQ')
%$$
%$$
%(e_2+e_3,SQ')(v,M)=(e_2+e_3+SQ'v,SQ'M)
%$$
%Take $v=e_1+e_2+e_3$, which is a valid eigenvector of $M$.

\subsection{$F \simeq Q_8$}\label{q8}

Let us determine quaternion subgroups of the different holomorphs. Let us observe that in this case we never have kernels of order $2$ since the unique normal subgroup of order $2$ is generated by the square of an element of order $4$.
On the other hand, we neither have Klein kernels since in a quaternion group the three different subgroups of order 4 are cyclic. Its conjugacy by automorphisms is the unique condition  we need to classify semidirect products. We denote as usual $i,j$ for two order 4 elements generating $Q_8$.

When $E=\Z_8$ in $\Hol(E)$ we have just a regular quaternion subgroup, which is normal, therefore union of conjugacy classes. It is
$$
F=\langle i=(2,1), j=(1,3) \ \rangle.
$$
We have possible cyclic kernels
$$
 \langle (2,1)\rangle,\
 \langle (1,3)\rangle\
\mbox{ and} \
\langle (7,3) \rangle.
$$
The first one cannot be conjugate to the other ones because the elements do not have the same second component.  The second and third ones are conjugate by $\Phi_7$.

\begin{proposition}
 Let $p=5$ or $p\ge 11$ be a prime. There are $3$ left braces with additive group
     $\Z_p\times \Z_8$ and multiplicative group $\Z_p\rtimes Q_8$.
     One is a direct product and the other two have cyclic kernel of order 4.
\end{proposition}

Now we look for the unique conjugacy class in $\Hol(\Z_4\times \Z_2)$ of regular subgroups isomorphic to $Q_8$.
The $``-1"$ element has to be the invariant element $((2,0),id)$. Therefore, the elements of order $4$ should have either $id$ or an element of order two in its second component.
We take the order $4$ element  $i=((1,0),id)$ so that $i^2=((2,0),id)$. Then $j=((0,1),rs)$ satisfies
$$
j^2=((2,0), id),\ k=ij=((1,1), rs),\ k^2=((2,0), id).
$$
Therefore the regular subgroup is $F=\langle ((1,0),id),\ ((0,1),rs)\rangle$. The possible cyclic kernels are
$$
 \langle ((1,0),id)\rangle,\
 \langle ((0,1),rs)\rangle\
\mbox{ and} \
\langle ((1,1),rs) \rangle.
$$
As we can see in the table of conjugacy classes in subsection \ref{z4z2} these cyclic subgroups are not conjugate.

\begin{proposition}
 Let $p=5$ or $p\ge 11$ be a prime. There are $4$ left braces with additive group
     $\Z_p\times \Z_4\times \Z_2$ and multiplicative group $\Z_p\rtimes Q_8$.
     One is a direct product and the other three have cyclic kernel of order 4.
\end{proposition}

Finally, inside $\Hol(\Z_2\times \Z_2\times \Z_2)$ we have also a unique conjugacy class of subgroups isomorphic to $Q_8$. %of length 14.
We should take elements $(v,A)$ with $A^2=Id$ and $v$ not an eigenvector. Then $(v,A)^2=(u,Id)$ with $u$ the unique non-zero vector in $Im(A+Id)$. Keeping the previous notations, we take the order $4$ element  $i=(e_3,S)$ so that $i^2=(e_2,Id)$. Then $j=(e_1,M)$ satisfies
$$
j^2=(e_2, Id),\ k=ij=(e_1+e_3, SM),\ k^2=(e_2, Id).
$$
Therefore the regular subgroup is $F=\langle (e_3,S),\ (e_1,M)\rangle$. The possible cyclic kernels are
$$
 \langle (e_3,S)\rangle,\
 \langle (e_1,M)\rangle\
\mbox{ and} \
\langle (e_1+e_3,MS) \rangle.
$$
Let us take the matrix of order $3$
$$
D:=\begin{bmatrix}
1&0&1\\
0&1&0\\
1&0&0\\
    \end{bmatrix}.
$$
Then $DSD^{-1}=M$ and $\Phi_D(e_3,S)=(e_1,M)$. Also $DMD^{-1}=MS$ and $\Phi_D(e_1,M)=(e_1+e_3,MS)$. This proves that $F$ is invariant under $\Phi_D$ and that the three cyclic subgroups are conjugate.

\begin{proposition}
 Let $p=5$ or $p\ge 11$ be a prime. There are $2$ left braces with additive group
     $\Z_p\times \Z_2\times \Z_2\times \Z_2$ and multiplicative group $\Z_p\rtimes Q_8$.
     One is a direct product and the other one has cyclic kernel of order 4.
\end{proposition}

\section{Total numbers}

For an odd prime $p\ne 3,7$ we compile in the following tables the total number of left braces of size $8p$.
Recall that for $p=3,7$ this number is given in \cite{V} and is 96 and 91, respectively.

The
additive group is $\Z_p\times E$ and the multiplicative group is a semidirect product $\Z_p\rtimes F$.
In the first column we have the possible $E$'s and in the first row the possible $F$'s.

\begin{itemize}
    \item If $p\ge 11$ and $p\not\equiv 1\bmod 4$
\begin{center}
\begin{tabular}{r|c|c|c|c|c|c}
&$\Z_8$   &   $\Z_4\times \Z_2$   &  $\Z_2\times \Z_2\times \Z_2$   & $D_{2\cdot 4}$ &$Q_8$ & \\
\hline
   $\Z_8$   &4 & 4& 0& 3&3&14\\
  $\Z_4\times \Z_2$ &0 &9 &8  &20 &4&41\\
  $\Z_2\times \Z_2\times \Z_2$  &0 &20 & 5& 8 &2&35\\
  \hline
  &4& 33 & 13 & 31 & 9 & $\mathbf{90}$\\
\end{tabular}
\end{center}

\item If $p\equiv 5\bmod 8$

\begin{center}
\begin{tabular}{r|c|c|c|c|c|c}
&$\Z_8$   &   $\Z_4\times \Z_2$   &  $\Z_2\times \Z_2\times \Z_2$   & $D_{2\cdot 4}$ &$Q_8$ & \\
\hline
   $\Z_8$   &6 & 6& 0& 3&3&18\\
  $\Z_4\times \Z_2$ &0 &13 &8  &20 &4&45\\
  $\Z_2\times \Z_2\times \Z_2$  &0 &28 & 5& 8 &2&43\\
  \hline
  &6& 47 & 13 & 31 & 9 & $\mathbf{106}$\\
\end{tabular}
\end{center}

\item If $p\equiv 1\bmod 8$

\begin{center}
\begin{tabular}{r|c|c|c|c|c|c}
&$\Z_8$   &   $\Z_4\times \Z_2$   &  $\Z_2\times \Z_2\times \Z_2$   & $D_{2\cdot 4}$ &$Q_8$ & \\
\hline
   $\Z_8$   &8 & 6& 0& 3&3&20\\
  $\Z_4\times \Z_2$ &0 &13 &8  &20 &4&45\\
  $\Z_2\times \Z_2\times \Z_2$  &0 &28 & 5& 8 &2&43\\
  \hline
  &8& 47 & 13 & 31 & 9 & $\mathbf{108}$\\
\end{tabular}
\end{center}

\end{itemize}

\end{document}